\documentclass[12pt]{article}

 \usepackage{graphics}
 \usepackage{graphicx}
 \usepackage{epsfig}
 \usepackage[all]{xy}
 \usepackage{cite}

 \usepackage{amssymb}
 \usepackage{amsthm,amsmath}

 \usepackage{latexsym, epsfig, psfrag,eepic,colordvi,bm}
 \usepackage{graphics,color}
 \usepackage{amsmath,amsfonts,amssymb,amscd}
 \usepackage[all]{xy}
 \usepackage{epstopdf}
 \usepackage{mathrsfs}

 \usepackage{amsthm,amsmath,amssymb}

 \usepackage{graphicx}
\usepackage{amsthm,amsmath,amssymb}

\usepackage{graphicx,epsfig}

\usepackage[colorlinks=true,citecolor=black,linkcolor=black,urlcolor=blue]{hyperref}
\usepackage{arydshln}

\theoremstyle{plain}
\newtheorem{theorem}{Theorem}[section]
\newtheorem{lemma}[theorem]{Lemma}
\newtheorem{corollary}[theorem]{Corollary}

\theoremstyle{definition}
\newtheorem{definition}{Definition}[section]

\theoremstyle{remark}
\newtheorem{remark}[theorem]{Remark}

\date{}

\title{\bf New equidistributions on plane trees and decompositions of $132$-avoiding permutations}

\author{Zi-Wei Bai, Ricky X. F. Chen\\
	\small School of Mathematics, Hefei University of Technology\\[-0.8ex]
	\small Hefei, Anhui 230601, P.~R.~China\\[-0.8ex]
	\small\tt 2021111458@mail.hfut.edu.cn, xiaofengchen@hfut.edu.cn
}

\begin{document}
\maketitle
\noindent{\bf Abstract.}
Our main results in this paper are new equidistributions on plane trees and $132$-avoiding permutations, two closely related objects.
As for the former, we discover a characteristic for vertices of plane trees that is equally
distributed as the height for vertices.
The latter is concerned with
four distinct ways of decomposing a $132$-avoiding permutation
into subsequences.
We show combinatorially that the subsequence length distributions of the four decompositions are mutually equivalent, and
there is a way to group the four into two groups such that each group is symmetric and the joint length distribution of one group
is the same as that of the other.
Some consequences are discussed.
For instance, we provide a new refinement of the equidistribution of internal vertices and leaves,
 and present new sets of $132$-avoiding permutations that are counted by the Motzkin numbers and their refinements.

\vskip 10pt



\section{Introduction}

Permutations with or without certain patterns have been extensively studied since Knuth's work~\cite{Knuth}.
Let $[n]=\{1,2,\ldots, n\}$ and $\mathfrak{S}_n$ be the symmetric group of permutations on $[n]$. Let $\tau=\tau_1 \tau_2 \cdots \tau_m \in \mathfrak{S}_m$ with $m\leq n$. A permutation $\pi=\pi_1\pi_2 \cdots \pi_n \in \mathfrak{S}_n$ is said to have a pattern $\tau$ if there exists a subsequence
$\pi_{i_1} \pi_{i_2} \cdots \pi_{i_m}$ of $\pi$ such that $\pi_{i_j} < \pi_{i_k}$ if and only if $\tau_j < \tau_k$. If $\pi$ does not have the pattern $\tau$, $\pi$ is called $\tau$-avoiding.
The permutation $\pi$ is said to have a consecutive pattern $\tau$ if there exists a subsequence
$\pi_{i} \pi_{i+1} \cdots \pi_{i+m-1}$ of $\pi$ that provides an occurence of the pattern $\tau$.

It is well understood that
the number of permutations on $[n]$ avoiding a pattern $\tau$ of length three is given by the Catalan number $C_n=\frac{1}{n+1}{2n \choose n}$ for any $\tau \in \mathfrak{S}_3$.
It is also very well known that $C_n$ counts plane trees of $n$ edges and respectively Dyck paths of semilength $n$.
A bijection between $132$-avoiding permutations and plane trees was given in Jani and Rieper~\cite{jani-rieper},
while a bijection between $132$-avoiding permutations and Dyck paths was given in Krattenthaler~\cite{krattenthaler}.
We refer to Claesson and Kitaev~\cite{CK} and references therein for more detailed discussion on bijections related to permutations avoiding a length three pattern.

This paper is mainly concerned with the set $\mathfrak{S}_n(132)$ of $132$-avoiding permutations and plane trees.
We examine four types of decompositions of permutations into subsequences.
When restricted to $132$-avoiding permutations, two of the four decompositions refine ascents and others
refine descents.
In fact, two of them are respectively increasing run and decreasing run decompositions which
have been studied, for instance, in Zhuang~\cite{yzhuang}, and Elizalde and Noy~\cite{eli-noy}.
One of our main results states that the length distributions of the subsequences of the four decompositions are mutually equivalent, and
there is a way to group the four into two groups such that each group is symmetric and the joint length distribution of one group
is the same as that of the other.
We prove this combinatorially, connecting several bijections, some are well-known and
some are recently discovered or new.
As a consequence, we are able to enumerate $132$-avoiding permutations according to a variety of filtrations.
For instance, taking advantage of a result of Elizalde and Mansour~\cite{sergi-toufik}, we present four sets of $132$-avoiding permutations that are respectively counted by the
Motzkin numbers. We additionally carry out some refined enumeration of these sets.

One employed new bijection between plane trees and $132$-avoiding permutations also allows us
to derive a new equidistribution result on plane trees concerning height for vertices.
As a corollary, we immediately recover the well-known fact that internal vertices and leaves are equidistributed.

The paper is organized as follows.
In Section~\ref{sec2}, we introduce the four types of decompositions and present some basic properties.
In Section~\ref{sec3}, several relevant bijections are reviewed. 
In Section~\ref{sec4}, we present a bijection between plane trees and $132$-avoiding permutations that appears
new.
As a consequence,
we introduce the right spanning width of vertices and show this new characteristic is equally distributed
as the height of vertices.
Finally, we prove the equidistribution result of the four decompositions and
provide a number of enumerative results as applications in Section~\ref{sec5}.
For instance, we present four sets of permutations that are all counted by the Motzkin numbers.

\section{Decompositions of permutations}\label{sec2}

For a permutation treated as a sequence, we have many ways to decompose it into different subsequences. Here we are interested in four
distinct decompositions which will be introduced in order.
The four decompositions can be viewed as refinements of the well studied statistics ascents and descents of permutations
as we shall see shortly.
Let $\pi=\pi_{1}\pi_{2}\cdots \pi_{n} \in \mathfrak{S}_n$.
An ascent of $\pi$ is an index $1\leq i<n$ such that $\pi_i < \pi_{i+1}$.
The rest are called descents of $\pi$.
Note that here we always treat $n$ as a descent.

\subsection{Increasing and decreasing run decompositions}

\begin{definition}
Suppose $\pi=\pi_{1}\pi_{2}\cdots \pi_{n} \in \mathfrak{S}_n$. A subsequence $\pi_i \pi_{i+1}\ldots \pi_{i+k-1}$
is called an \emph{increasing run} (IR\footnote{We will write IR (DR and v-CIS defined later)
	in singular as well as plural form.}) if $\pi_i<\pi_{i+1}<\cdots <\pi_{i+k-1}$, and it is not contained in a longer such subsequence.
\end{definition}

Obviously, a permutation $\pi$ can be uniquely decomposed into IR. This decomposition
will be simply referred to as IRD.
For example, we can decompose a permutation $\pi=5346127$ into three IR: 
$$
\tau_{1}=5, \quad \tau_{2}=346, \quad \tau_{3}=127.
$$
An integer partition $\lambda$ of $n$, denoted by $\lambda \vdash n$, is a sequence of non-increasing positive integers $\lambda=\lambda_1\lambda_2\cdots \lambda_l$
such that $\sum_i \lambda_i=n$.
The lengths of the IR of $\pi$ give the length distribution of $\pi$ w.r.t.~IRD.
We will encode the length distribution by an integer partition of $n$.
For example, the length distribution of $\pi=5346127$ w.r.t.~IRD is given by $\lambda=331$.
Note that if $i$ is a descent of $\pi$, then $\pi_i$ is the rightmost (or last) element of an IR while
$\pi_{i+1}$ (if $i<n$) starts a new IR of $\pi$.
Thus, there is an obvious one-to-one correspondence between IR and descents.
Consequently, the set of permutations with $k$ descents can be further refined into subsets by the length distribution of the corresponding $k$ IR.

\emph{Decreasing runs} (DR) and decreasing run decomposition (DRD) are defined analogously.
It is also apparent that the number of ascents of $\pi$ plus one is the same as the number of segments from the DRD of $\pi$.
Consequently, DR with the associated length distribution may be viewed as a refinement of ascents.

\subsection{Value-consecutive increasing subsequences}

An increasing run can be alternatively interpreted as a maximal position-consecutive
increasing subsequence. Then, it suggests a natural counterpart which may be called maximal \emph{value-consecutive increasing subsequences} (v-CIS).
\begin{definition}
A v-CIS {of a permutation $\pi=\pi_{1}\pi_{2}\cdots \pi_{n}$} is a subsequence of the form $\pi_{i_1} \pi_{i_2}\cdots \pi_{i_k}=j(j+1)\cdots (j+k-1)$.
\end{definition}

The decomposition of a permutation $\pi$ into its maximal v-CIS is also referred to as v-CIS (of $\pi$).
Taking $\pi=5346127$ as an example, its v-CIS decomposition gives subsequences: 
$$
567,\quad 34, \quad 12,
$$
the length distribution of which is given by the partition $322$.

\begin{lemma} \label{lem:v-descent}
	The number of descents of $\pi \in \mathfrak{S}_n(132)$ equals the number of segments from its v-CIS.
	In particular, if $i\neq n$ is a descent of $\pi$, then
	$\pi_{i+1}$ starts a maximal v-CIS of $\pi$.
\end{lemma}

\begin{proof}
	Let $\pi=\pi_1\pi_2\cdots \pi_n \in \mathfrak{S}_n(132)$. It suffices to show that there is a one-to-one correspondence
	between the descents $i\neq n$ of $\pi$ and the maximal v-CIS of $\pi$ which
	do not start with $\pi_1$. First, if $i\neq n$ is a descent of $\pi$, then we claim that
	$\pi_{i+1}$ starts a maximal v-CIS of $\pi$. Otherwise, $\pi_j=\pi_{i+1}-1$ for some $j<i$.
	In this case, $\pi_j \pi_i \pi_{i+1}$ yields a $132$ pattern, a contradiction.
	Conversely, suppose $\pi_{i+1}$ ($i>0$) starts a maximal v-CIS.
	If $\pi_{i+1}=1$, then obviously $i$ is a descent;
	If $\pi_{i+1}\neq 1$, then $\pi_j=\pi_{i+1}-1$ for some $j>i+1$ due to the maximality of the v-CIS containing $\pi_{i+1}$.
	Consequently, $\pi_i<\pi_{i+1}$ implies $\pi_i< \pi_{i+1}-1=\pi_j$ ,
	which again yields a $132$ pattern $\pi_i \pi_{i+1} \pi_j$.
	Thus, we always have $\pi_i> \pi_{i+1}$ whence $i$ is a descent of $\pi$.
	In view that there is a maximal v-CIS starting with $\pi_1$
	and $n$ is a descent of $\pi$, the lemma follows.
\end{proof}

We remark that the above relation is not true for a general permutation. For instance, $\pi=153642$ clearly has a $132$ pattern. We
can check that $\pi$ has four descents but only
three segments in its v-CIS.
This is actually interesting as the other three decompositions studied in this paper
are directly related to ascents and descents for general permutations, and may deserve future investigations.

\begin{lemma}\label{lem:nest}
	Let $\pi=\pi_1\pi_2\cdots \pi_n$ be a $132$-avoiding permutation. Then, there do not exist $1\leq i<j<k<l\leq n$ such that $\pi_i$ and $\pi_k$ are contained in the same
	v-CIS $\tau_1$ while $\pi_j$ and $\pi_l$ are contained in the same v-CIS $\tau_2$ where $\tau_1 \neq \tau_2$.
\end{lemma}
\proof Suppose such $i,j,k,l$ exist. By construction, if $\pi_j>\pi_i$, then $\pi_j>\pi_k$ as well
since otherwise the three must be contained in the same v-CIS.
Thus, $\pi_i \pi_j\pi_k$ provides a $132$ pattern, a contradiction.
Analogously, if $\pi_j<\pi_i$ (so $\pi_j< \pi_k$), we then have $\pi_l<\pi_k$,
which implies $\pi_j \pi_k\pi_l$ being a $132$ pattern.
Either way yields a contradiction and the lemma follows. \qed

\begin{lemma}\label{lem:noncrossing}
	Let $\pi=\pi_1\pi_2\cdots \pi_n \in \mathfrak{S}_n(132)$.
	Suppose $\tau$ and $\tau'$ are two distinct v-CIS of $\pi$.
	Then, either all elements in $\tau$ lie between two consecutive elements in $\tau'$,
	or all elements in $\tau$ are to the left of the starting element of $\tau'$.
	Moreover, in the former case, the maximal element in $\tau$ is smaller than the minimal element of $\tau'$,
	while in the latter case, the maximal element in $\tau'$ is smaller than the minimal element in $\tau$.
\end{lemma}
\proof The first statement follows from Lemma~\ref{lem:nest}.
In the remaining part, the ``former" case is true because the minimal (i.e., starting) element
of $\tau$ is the image of an element determining a descent in view of Lemma~\ref{lem:v-descent};
the ``latter" case is true since otherwise an element from $\tau$, $\pi_j$ and an element from $\tau'$
form a $132$ pattern, where $\pi_j$ determines the descent corresponding to $\tau'$ in the light of Lemma~\ref{lem:v-descent}.
This completes the proof. \qed

\subsection{Layered decreasing envelops}

{
\begin{definition}	
	Let $\pi=\pi_1 \pi_2 \cdots \pi_n$ be a permutation.
	A right-to-left maximum w.r.t. a position $m$ (or starting with $\pi_m$) is an entry $\pi_i$ such that
	$\pi_i > \pi_j$ for all $i< j \leq m$, and $\pi_m$ is always viewed as a right-to-left maximum.
	\end{definition}
	
	Let $\pi=\pi_1 \pi_2 \cdots \pi_n$ be a permutation. Starting with $\pi_n$, we search all right-to-left maxima in $\pi$ which
	will give us a decreasing subsequence of $\pi$ (i.e., a kind of decreasing path);
	Next starting with the position that preceeds the leftmost element of the last found path, we search all
	right-to-left maxima which will give us a new decreasing path;
	Continue doing this until we get a path starting with $\pi_1$.
	At this point, we have found the ``outermost" layer of decreasing paths.
	Imagine we have connected adjacent elements in the same decreasing path by edges.
	Repeat the procedure of obtaining decreasing paths with respect to the segment of entries covered by the same edge in the existing path
	until all entries in $\pi$ have been placed into a path.
	See an example for $\pi=10,8,7,9,11,6,4,3,5,12,1,2$ in Figure~\ref{fig:lde}.
	This decomposition of permutation elements into decreasing subsequences will be called layered decreasing envelop
	decomposition, simply referred to as LDE.

	Apparently, if $i\neq n$ is an ascent of $\pi$, then $\pi_i$ is the rightmost element
	of a decreasing path and vice versa. Taking account of the decreasing path ending with $\pi_n$, the number of paths
	is one greater than the number of ascents of $\pi$.
}
	We will be interested in the length distribution of these decreasing paths.
	For example, the LDE length distribution of $\pi$ in Figure~\ref{fig:lde}
	is the partition $32222$.

\begin{figure}[h]
	\centering
	\includegraphics[width=0.6\columnwidth]{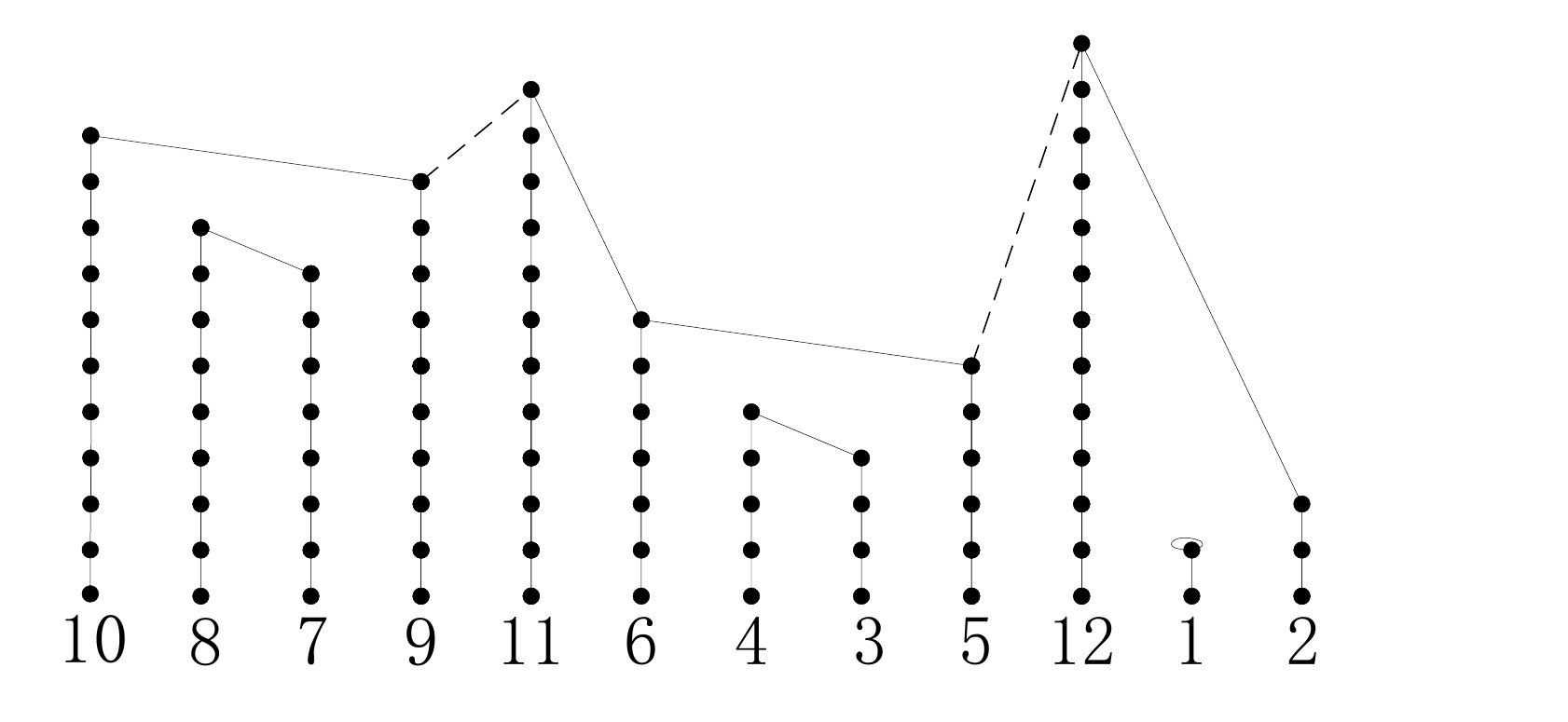}
	\caption{The LDE of a $132$-avoiding permutation, where elements belonging to the same decreasing subsequence are connected by solid lines. }
	\label{fig:lde}
\end{figure}

It is also worth noting that by construction
any two decreasing paths are either in a left-right position or in a covering (or nesting) relation.

\section{Relevant existing bijections}\label{sec3}

In this section, we review several bijections involving plane trees and  $132$-avoiding permutations which can be
found in the literature and will be used later.

\subsection{A bijection $\varphi$ on plane trees}

A \emph{plane tree} $T$ can be recursively defined as an unlabeled tree with one distinguished vertex called the \emph{root} of $T$, where the unlabeled trees resulting from the removal of the root as well as its adjacent edges from $T$ are linearly ordered, and they are plane trees with the vertices in $T$ adjacent to the root of $T$ as their respective roots.  In a plane tree $T$, the number of edges in the unique path from a vertex $v$ to the root of $T$ is called the \emph{height} (or \emph{level}) of $v$, and the vertices adjacent to $v$ but having greater heights are called the \emph{children} of $v$.
The number of children of $v$ is called the \emph{outdegree} of $v$. The vertices on level $2i$ (resp.~$2i+1$) for $i\geq 0$ are called even-level (resp.~odd-level) vertices. A non-root vertex without any child is called a \emph{leaf}, and an \emph{internal vertex} otherwise. The root is always treated as an internal vertex.
We will draw a plane tree with its root on the top level, i.e., level $0$, and with the children of a level $i$ vertex arranged on level $i+1$ (below level $i$) left-to-right following their linear order.

A plane tree can be decomposed into a set of paths where each path has a leaf as a terminal vertex.
There are two ways to achieve that: left path decomposition and right path decomposition. 
\begin{definition}[Left path decomposition]
Suppose all leaves in a plane tree $T$ are ordered by their relative order in the depth-first search of $T$ from left to right. The first path is the path from the first leaf to the root, and
for $t>1$, the $t$-th path goes from the $t$-th leaf up to the first vertex that is already in
a path that has been obtained.
\end{definition}
We will call the multiset consisting of the lengths of the obtained paths the \emph{left path distribution} of the given tree. The right path decomposition is analogous. That is, the paths are successively
obtained from right to left.

RNA plays an important role in various biological processes within a cell.
In particular, RNA secondary structures have been extensively studied from a computational and combinatorial
viewpoint, for the purposes of structure and function predictions. We refer to Smith and Waterman~\cite{waterman4}, and Schmitt and Waterman~\cite{waterman} for the definition and discussion.
Two bijections between RNA secondary structures and plane trees are relevant, one being the Schmitt-Waterman bijection~\cite{waterman}
and the other being the new bijection recently discovered by Chen~\cite{chen3}.
Clearly, with RNA secondary structures as intermediates, we have a bijection $\varphi$ from plane trees to plane trees.
We refer to~\cite{waterman} and~\cite{chen3} for details about the bijections.
What one really need in this paper is the following property of the bijection $\varphi$.

\begin{theorem}[Chen~\cite{chen3}]\label{thm:chen3}
	The bijection $\varphi$
maps a plane tree of $n>0$ edges with $x$ internal vertices of outdegree $q$ and $y$ (left) paths of length $l$ in its left path decomposition
to a plane tree of $n$ edges with $x$ odd-level vertices of outdegree $q-1$ and $y$ even-level vertices
of degree $l$, where in particular, the length of the first (leftmost) path in the former equals the degree of the root of the latter.
\end{theorem}

\subsection{The Jani-Rieper bijection}

An explicit bijection between plane trees and $132$-avoiding permutations was given by Jani and Rieper~\cite{jani-rieper}.
The following is how it works.
Let $T$ be a plane tree of $n$ edges. We use a preorder traversal of $T$ (from left to right) to label the non-root vertices in decreasing order with the integers $n, n-1,\ldots, 1$. As such, the first vertex visited gets the label $n$ and the last receives 1. A permutation written as a word is next obtained by reading the labelled tree in postorder, that is, traverse the tree from left to right and record the label of a vertex when it is last visited.

The reverse from a $132$-permutation to a plane tree was not explicitly presented in Jani and Rieper~\cite{jani-rieper}.
Here we present a procedure and we leave it to the reader to verify its effectiveness.
Let $\pi$ be a $132$-avoiding permutation.
Suppose the IR of $\pi$ from left to right are $\tau_1, \tau_2, \ldots, \tau_k$.
Start with $\tau_k$ and make it into a path with the maximal (i.e., rightmost) element in $\tau_k$ attaching to the root of the expected plane tree $JR(\pi)$.
For example, suppose $\pi=10, 8, 7, 9, 11, 6, 4, 3, 5, 12, 1 ,2$. Then, $\tau_k=12$ and the path will be the path from vertex $1$ to the root of the left tree
in Figure~\ref{fig:new-bij}. After $\tau_i$ is ``placed" in the (partial) tree, we find the minimal element $u$ in the leftmost
path (i.e., the one from the leftmost leaf to the root) in the current partial tree that is larger than the maximal element $x$ in $\tau_{i-1}$,
and attach the path induced by $\tau_{i-1}$ to the tree such that $u$ and $x$ are adjacent;
if no such a $u$ exists, we attach the path induced by $\tau_{i-1}$ to the root of the current tree.
Eventually, we obtain a plane tree.
In the following, we will regard the corresponding plane trees of $132$-avoiding permutations as plane trees with vertex labels,
although the vertex labels are uniquely determined (by the underlying bijections) and can be omitted.
Moreover, we may sometimes use the label of a vertex to refer to the vertex.

\begin{lemma}[Jani-Rieper~\cite{jani-rieper}]
 The longest increasing subsequence in $\pi \in \mathfrak{S}_n(132)$ starting with a leaf is the height of the leaf in $JR(\pi)$.
\end{lemma}

\begin{lemma}\label{lem:JR-distribution}
For $\pi \in \mathfrak{S}_n(132)$, the outdegree distribution of the internal vertices of $JR(\pi)$ is the same as
	the LDE length distribution of $\pi$ while the right path distribution of $JR(\pi)$ is the same as the IR length distribution of $\pi$.
\end{lemma}

\proof By construction, the descendants (if any) of a non-root vertex $v$ in $JR(\pi)$ are smaller than $v$ and in $\pi$ the
descendants appear before $v$ but after any siblings to the left of $v$ in the tree $JR(\pi)$;
For the root of $JR(\pi)$, its children decrease from left to right and their left-to-right
relative order in $\pi$ agrees with the one in $JR(\pi)$.
In particular, the rightmost child of the root is $\pi_n$.
Thus, the outermost and rightmost decreasing path of the LDE of $\pi$ consists of the children of the root of $JR(\pi)$.
Analogously, the children of the root of a subtree of $JR(\pi)$ determine an LDE decreasing path.
Iteration of the reasoning gives rise to the correspondence between the outdegree distribution and LDE length distribution.
See Figure~\ref{fig:lde} and~\ref{fig:new-bij}.
The ``while" part follows from our reverse procedure from $132$-avoiding permutations to plane trees. \qed

\subsection{Odd-even level switching of plane trees}

Given a plane tree $T$, we obtain a new plane tree $T'$ by taking the leftmost child $v$ of the root of $T$ as the root of the new tree $T'$, i.e., lifting $v$ to the top level such that the even-level vertices in $T$ become odd-level vertices in $T'$ and vice versa. This
is clearly a bijection. In addition, it is not difficult to verify that the degree distribution of the even-level (resp.~odd-level) vertices of $T$ becomes the degree distribution of the odd-level (resp.~even-level) vertices of $T'$.

\section{A new bijection and consequences}\label{sec4}

In this section, we first present a bijection between plane trees and $132$-avoiding permutations
which appears to be new. Then, we discuss several applications of the bijection, in particular,
a new equidistribution result on plane trees.

Let $\pi=\pi_{1}\pi_{2}\cdots \pi_{n} \in \mathfrak{S}_n(132)$. Suppose there are $k$ segments in its v-CIS: $\tau_{1}=\pi_{1}^{1}\pi_{2}^{1}\cdots\pi_{i_{1}}^{1},\ldots,\tau_{k}=\pi_{1}^{k}\pi_{2}^{k}\cdots\pi_{i_{k}}^{k}$,
and suppose for $1 \leq i \leq k$,
$$
\pi_1^i=\max\{\pi_1^i, \pi_1^{i+1},\ldots, \pi_1^k\}.
$$
We then construct a plane tree $T=\phi(\pi)$ recursively with the following procedure.
\begin{itemize}
\item First, start with a single vertex and arrange $\pi_{1}^{1},\ldots, \pi_{i_{1}}^{1}$ from left to right as the children of the
vertex.
\item For $j=2$ to $k$, find vertex $\pi_t$ that has already
appeared in the so-far constructed partial tree and is immediately to the left of $\pi_1^j$ in $\pi$; then
let $\pi_{1}^{j},\ldots, \pi_{i_{j}}^{j}$ be the left-to-right children of vertex $\pi_t$.
\end{itemize}
Noticing that the above desired $\pi_t$ always exist due to Lemma~\ref{lem:v-descent} whence we eventually obtain a labelled plane tree $T$.
It is also not difficult to observe that there are $k$ internal vertices in $T$.

For example, $\pi=10 ,8, 7, 9, 11, 6, 4 ,3, 5 ,12, 1, 2$  is a 132-avoiding permutation, and its v-CIS are
$$
\tau_{1}=10, 11, 12 ,\quad \tau_{2}=8 9, \quad \tau_{3}=7, \quad \tau_{4}=6, \quad \tau_{5}=4 5,\quad \tau_{6}=3,\quad \tau_{7}=1 2.
$$
Then, its corresponding labelled plane tree is depicted in Figure~\ref{fig:new-bij} (right).

We will next show that the vertex labels in $T$ are completely determined
by its underlying unlabelled plane tree.
Suppose $u$ and $v$ are two distinct non-root vertices in $T$.
In view of Lemma~\ref{lem:noncrossing}, we have the following two properties:
\begin{itemize}
	\item[(i)] If $u$ is on the path from $v$ to the root of $T$, then the minimum child
	of $u$ is larger than the maximum
	child (if any) of $v$. The reason is, in terms of Lemma~\ref{lem:noncrossing}, the v-CIS $\tau$ corresponding to the children of $v$ lie between
	two consecutive elements of the v-CIS $\tau'$ corresponding to the children of $u$.
	\item[(ii)] If $v$ is a descendant of a vertex on the path from $u$ to the root of $T$ (but
	not a descendant of $u$) and in a sense on the right-hand side of the path, then the minimum child (if any) of $u$ is larger than
	the maximum child (if any) of $v$. The reason for this case is, in terms of Lemma~\ref{lem:noncrossing}, all of the elements in the v-CIS $\tau$ corresponding to the children of $u$ are to the left of the leftmost element of 
the v-CIS $\tau'$ corresponding to the children of $v$.
\end{itemize}
An example of case (i) in Figure~\ref{fig:new-bij} (right) is $u=11$ and $v=4$,
while an example of case (ii) in Figure~\ref{fig:new-bij} (right) is $u=6$ and $v=12$.

As a result, it may not be hard to conclude that the labels of the vertices are uniquely determined by the underlying unlabelled plane tree
as follows: if we travel the internal vertices of the plane tree in the left-to-right depth-first manner, then
the $k$ left-to-right children of the current internal vertex carry the remaining largest $k$ elements from $[n]$
in increasing order.
So, in theory, we can remove the vertex labels and only consider the underlying unlabelled plane tree.
But, we prefer to keeping the vertex labels as it is more convenient to refer to the vertices.

\begin{figure}[h]
\centering
\includegraphics[width=0.8\columnwidth]{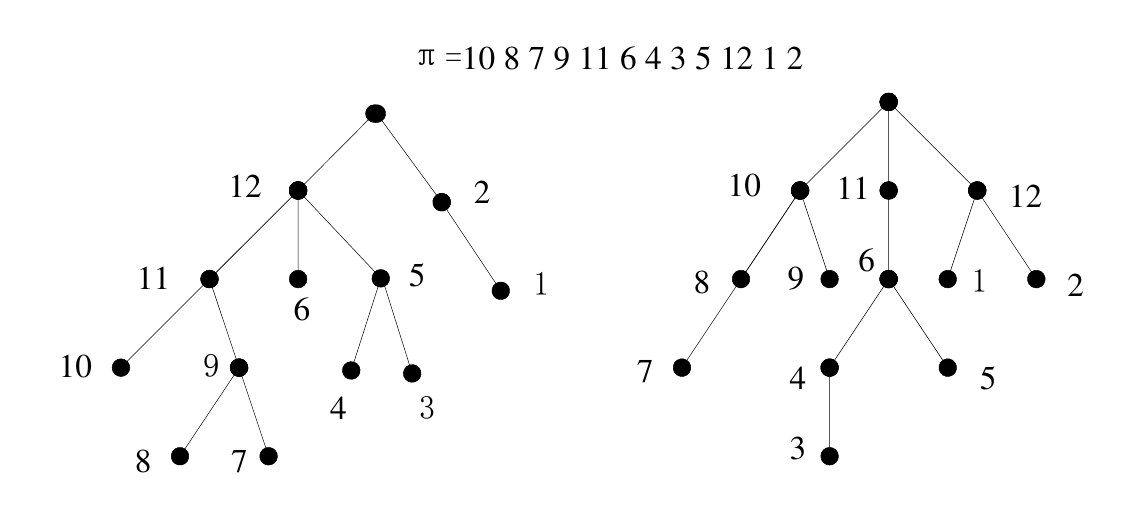}
\caption{The corresponding plane trees under the Jani-Rieper bijection (left) and the bijection $\phi$ (right) of the same $132$-avoiding permutation $\pi$.}
\label{fig:new-bij}
\end{figure}

As for the reverse, from a plane tree that is uniquely labelled in the manner just described right above, it is not difficult to
see that the left-to-right depth-first search (or preorder) gives us the desired $132$-avoiding permutation.
We leave the proofs of this and the following lemma to the reader.

\begin{lemma}\label{new-distribution}
	Given $\pi \in \mathfrak{S}_n(132)$, let $T=\phi(\pi)$.
	Then, the outdegree distribution of internal vertices of $T$ is the same as
	the v-CIS length distribution of $\pi$ while the left path distribution of $T$ equals the DR length
	distribution of $\pi$. In particular, the length of the DR starting with $\pi_1$ in $\pi$ is the same as the length
	of the first path in $T$.
\end{lemma}

As the first application of the bijection $\phi$, we obtain a new equidistribution result on plane trees
which involves a new quantity associated to vertices.

\begin{definition}[Right spanning width]
	Let $v$ be a vertex in a plane tree $T$.
The sum of the number of edges attached to other vertices (than $v$) on the path from $v$ to the root of $T$ from the right-hand side of the path
and the number of children of $v$ is called
the right spanning width of $v$, denoted by $rsw(v)$.
\end{definition}
See an illustration in Figure 4.
In addition, we define
\[
rsw(T)=\max\{rsw(v): \mbox{$v$ is an internal vertex in $T$}\}.
\]

\begin{figure}[h]
\centering
\includegraphics[width=0.28\columnwidth]{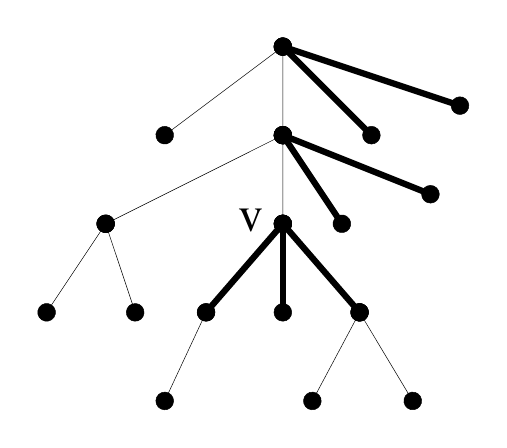}
\caption{The right spanning width of vertex $v$ is the number of bold edges.}
\end{figure}

\begin{theorem}\label{thm:rsw}
	The number of plane trees of $n$ edges the heights of whose vertices constitute a multiset $\mathcal{M}$ of $n+1$ elements
	is the same as the number of plane trees of $n$ edges the right spanning widths of whose vertices constitute $\mathcal{M}$.
	
	Moreover, the number of plane trees of $n$ edges and $k$ leaves whose heights constitute a multiset $\mathcal{M}'$ is the same
	as the number of plane trees of $n$ edges and $k$ internal vertices whose right spanning
	widths constitute $\mathcal{M}'$.
\end{theorem}
\begin{proof}
Let $\pi=\pi_1\pi_2\cdots \pi_n \in \mathfrak{S}_n(132)$.
Suppose $T=JR(\pi)$ and $T'=\phi(\pi)$.
First, it is not hard to observe that the label $\pi_i$ of a leaf in $T$ is the label of the leftmost
child of an internal vertex $v$ in $T'$ by construction.
Now consider the longest increasing subsequence $\tau$ starting with $\pi_i$.
Recall in terms of $T$, the length of $\tau$ is the height of the leaf $\pi_i$.
We next show that in terms of $T'$, the length of $\tau$ is $rsw(v)$.
According to $\phi$, the elements larger than $\pi_i$ and to the right of $\pi_i$
must be on the right-hand side of the path from $\pi_i$ to the root of $T'$.
Moreover, they must attach to some internal vertices on the path.
Consequentely, we can easily argue that the longest path starting with $\pi_i$
consists of vertices incident to the internal vertices on the path where there are $rsw(v)$
of them.

Next, for the label $\pi_i$ of an internal vertex in $T$, the longest increasing subsequence starting with $\pi_i$ in $\pi$
also equals the height of $\pi_i$ in $T$.
By construction, $\pi_i$ being of an internal vertex in $T$ implies that $i\neq 1$ and $\pi_{i-1}<\pi_i$,
i.e., $i-1$ is an ascent of $\pi$.
However, in $T'$, $\pi_{i-1}$ must be a leaf.
Similar to the above discussion, the longest increasing subsequence starting with $\pi_{i-1}$
has length exactly $rsw(\pi_{i-1})+1$ in $T'$.
Accordingly, the length of the longest increasing subsequence starting with $\pi_{i}$ equals $rsw(\pi_{i-1})$, where
$\pi_{i-1}$ is a leaf in $T'$, and the theorem follows.
\end{proof}

\begin{remark}
Recall the height of a plane tree is the maximum height of leaves in the tree.
While average height of various trees and a single leaf there were examined in a plethora of work,
see for instance, de Bruijin, Knuth and Rice~\cite{height72}, Kemp~\cite{hemp83}, and Prodinger~\cite{prodinger84},
to the best of our knowledge, no statistics have been found to be equidistributed as height before.
It is true that the right spanning width of a vertex appears not as natural as the height of a vertex.
However, Theorem~\ref{thm:rsw} also provides a valuable instance of equidistribution results that
are somehow different from most of existing equidistribution results on plane trees.
Namely, equidistribution results are usually in the form: a statistic $A$ (e.g., degree) over a subset of vertices, e.g., internal vertices (instead of all vertices) 
is equally distributed as a statistic $B$ over another subset of vertices.
This kind of results cannot be generalized to the form: $A$ and $B$ are equidistributed over all vertices.  
See such an example in Theorem~\ref{thm:chen3} (internal vertices vs odd-level vertices).

\end{remark}

\begin{remark}[Open problem]
	It is natural to define left spanning width ($lsw$) of a vertex analogously.
	We believe there exists another statistic $A$ associated to vertices such that the pair
	$(lsw, rsw)$ is in equidistribution with the pair of $A$ and height of vertices.
	However, we failed to find it.
\end{remark}

\begin{corollary}
The number of plane trees of $n$ edges and height $k$ equals the number of plane trees $T$ of $n$ edges with $rsw(T)=k$.
\end{corollary}

Apparently, the latter part of Theorem~\ref{thm:rsw} refines the following well-known fact.

\begin{corollary}
	The number of plane trees of $n$ edges with $k$ internal vertices is the same as the number of plane trees of $n$ edges with $k$ leaves.
\end{corollary}

It is well known that the number of Dyck paths of semilength $n$ with $k$ peaks
is given by the Narayana number $N(n,k)=\frac{1}{n}{n\choose k} {n\choose k-1}$ with $N(0,0)=1$.
In Callan~\cite{callan22}, it was proved via a novel combinatorial argument that the number of Dyck paths of semilength $n$ with $i$ returns to ground level and $j$ peaks is the generalized Narayana number 
$$
N_{i}(n,j)=\frac{i}{n}{n\choose j}{{n-i-1}\choose {j-i}}.
$$

\begin{theorem}
	The number of $132$-avoiding permutations on $[n]$ starting with $i$ and having $k$ descents is given by $\frac{n+1-i}{n}{{n}\choose {k-1}}{{i-2}\choose{i-k}}$.
	Furthermore,
	the number of $132$-avoiding permutations on $[n]$ starting with $i$, ending with $j$ and having $k$ descents is
	\begin{align}
		\begin{cases}
			N_{n-i}(n-1,n-k), \quad &\mbox{if $i<j$},\\
			\sum_{m=1}^{k-1}N_{n+1-i}(n-j,n-m+1)N(j-1,j+m-k+1), \quad &\mbox{else}.
		\end{cases}
	\end{align}
\end{theorem}

\begin{proof}
	
	For $\pi \in \mathfrak{S}_n(132)$ in question in the first part of the theorem, its v-CIS containing the first element $i$ is $i(i+1)\cdots n$.
	According to the bijection $\phi$, this implies the outdegree of the root of the tree $\phi(\pi)$ equals $n+1-i$.
	Recall that the number of descents of $\pi$ is equal to the number of v-CIS of $\pi$ which equals the number of
	internal vertices in $\phi(\pi)$. So, the number of
	leaves in $\phi(\pi)$ is $n+1-k$.
	It is well known that the set of plane trees of $n$ edges and with these features
	has the same size as the set of Dyck paths of semilength $n$ with $n+1-i$ returns to ground level and $n+1-k$ peaks.
	As a result, the first statement follows.

	As for the remaining part, we need the following {\em claim}: In a $132$-avoiding permutation ending with $j$, the last $j$ entries are from the set $[j]$.
	Otherwise, there is no difficulty to show there exists a $132$ pattern.
	
	In order to prevent the appearance of $132$ patterns, if $i<j$ then $j=n$.
	Then, the considered set is equivalent to the set of $132$-avoiding permutations on $[n-1]$ starting with $i$ and having $k$ descents which
	is given by $N_{n-i}(n-1,n-k)$ as in the first statement.
	
	Similarly, if $i>j$ and $j=1$, the considered set is equivalent to the set of $132$-avoiding permutations on $[n-1]$ starting with $i-1$ and having $k-1$ descents which
	is given by $N_{n+1-i}(n-1,n-k+1)$.

	If $i>j$ and $j\neq1$, from the above claim, any $132$-avoiding permutation in this case is a concatination of two $132$-avoiding
	permutations: one is on $[n]\setminus [j]$ which starts with $i$ and has $m$ descents, and the other
	is on $[j]$ which ends with $j$ and has $k-m$ descents, for some $1\leq m < k$.
	The former are equivalent to $132$-avoiding permutations
	on $[n-j]$ which start with $i-j$ and have $m$ descents,
	while the latter are equivalent to $132$-avoiding permutations on $[j-1]$ having $k-m$ descents
	which are in one-to-one correspondence with plane trees with $j$ vertices $k-m$ of which are internal via $\phi$.
	Consequently, the desired number is clearly
	$$
	\sum_{m=1}^{k-1}N_{n+1-i}(n-j,n-j-m+1)N(j-1,j+m-k).
	$$
	Note that $N(0,0)=1$ and $N(0,x)=0$ for $x>0$. The last quantity agrees to the number for the case $j=1$,
	and the proof follows.	
\end{proof}

\section{Equidistributions on $132$-avoiding permutations}\label{sec5}

Our main goal in this section is to prove the following new equidistribution result on $132$-avoiding permutations.
We point out that common tricks such as taking the inverse, taking the reverse (i.e., reading from right to left) or taking the complement (i.e., replacing $i$ with $n+1-i$)
will not work since the resulting permutations may not be $132$-avoiding permutations
anymore. Instead, we will connect several bijections carefully in order for proving
the theorem.

\begin{theorem}\label{thm:equi}
Given two partitions $\lambda, \mu \vdash n$, the following four sets are of equal size:
\begin{itemize}

\item[(1)] $\pi \in \mathfrak{S}_n(132)$ whose IRD and LDE length distributions are resp.~$\lambda$ and $\mu$.

\item[(2)]
$\pi \in \mathfrak{S}_n(132)$ whose IRD and LDE length distributions are resp.~$\mu$ and $\lambda$.

\item[(3)] $\pi \in \mathfrak{S}_n(132)$ whose v-CIS and DRD length distributions are resp.~$\lambda$ and $\mu$.

\item[(4)] $\pi \in \mathfrak{S}_n(132)$ whose v-CIS and DRD length distributions are resp.~$\mu$ and $\lambda$.
\end{itemize}
\end{theorem}

\begin{proof}
	We first prove the sets of (1) and (2) have the same size.
	Let $\pi\in \mathfrak{S}_n(132) $ and suppose its IRD and LDE length distributions are resp.~$\lambda$ and $\mu$.
	Let $T_1=JR(\pi)$, i.e., the corresponding tree of $\pi$ under the Jani-Rieper bijection. Denote by $T_2$ the mirror image of $T_1$ (i.e., horizontal flipping). Clearly, in view of Lemma~\ref{lem:JR-distribution}, the internal vertex outdegree distribution of $T_2$ is $\mu$ and the left path length distribution of $T_2$ is $\lambda$.
	Accordingly, following from Chen's bijection $\varphi$ (Theorem~\ref{thm:chen3}), the odd-level outdegree distribution of $T_3=\varphi(T_2)$ is $\mu-1$ (that is, $(\mu_1-1)(\mu_2-1)\cdots (\mu_k-1)$ for $\mu=\mu_1\mu_2\cdots \mu_k$) while the even-level degree distribution of $T_3$ is $\lambda$.
	
	Let $T_4$ be the resulting tree via the odd-even level switching transform from $T_3$.
	Then, it is easily seen that the odd-level outdegree distribution of $T_4$ is $\lambda-1$ while the even-level degree distribution of $T_4$ is $\mu$.
	 According to Theorem 3.1, the internal vertex outdegree distribution of $T_5=\varphi^{-1} (T_4)$ is thus $\lambda$ and the left path length distribution of $T_5$ is $\mu$.
	Consequently, the mirror image $T_6$ of $T_5$ has the internal vertex outdegree distribution $\lambda$ and the right path length distribution $\mu$.
	Thus, the corresponding permutation $\pi'$ of $T_6$ under the Jani-Rieper bijection has the IRD and LDE length distributions resp.~$\mu$ and $\lambda$. 
	
	See Figure~\ref{fig:equidistribution} for an illustration.
	Obviously, the correspondence between $\pi$ and $\pi'$ is a bijection as it is essentially the composition of a number of bijections.
	Hence, the sets of (1) and (2) are of equal size.
	\begin{figure}[h]
		\centering
		\includegraphics[width=0.8\columnwidth]{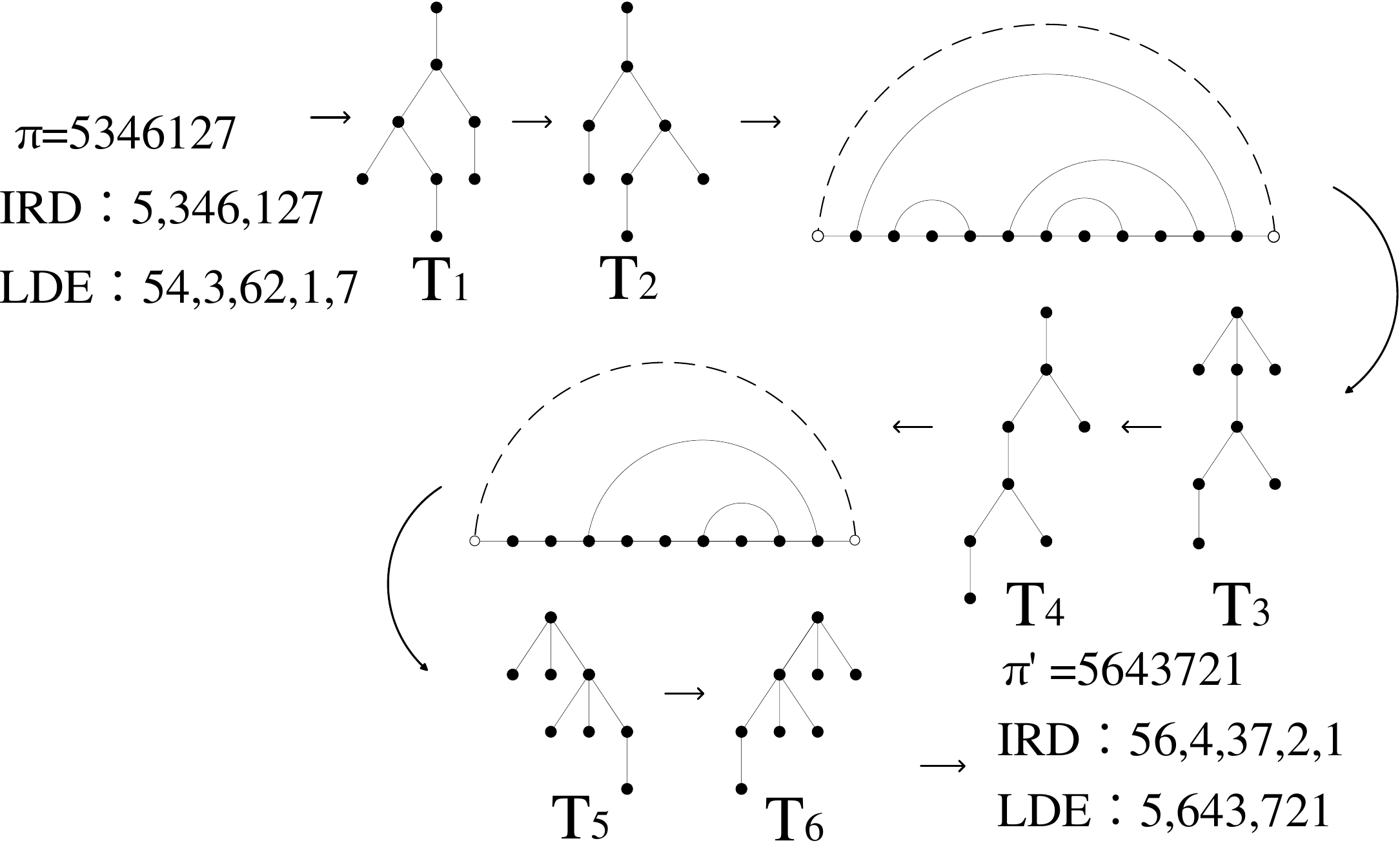}
		\caption{The correspondence between set of (1) and set of (2). }
		\label{fig:equidistribution}
	\end{figure}
	
	Next we prove the sets of (1) and (3) contain the same number of $132$-avoiding permutations.
	Again, let $\pi\in \mathfrak{S}_n(132) $ and suppose its IRD and LDE length distributions are resp.~$\lambda$ and $\mu$.
	Let $T_1=JR(\pi)$, and denote by $T_2$ the mirror image of $T_1$.
	 Then, the internal vertex outdegree distribution of $T_2$ is $\mu$ and the left path length distribution of $T_2$ is $\lambda$.
	 Suppose $\pi'$ is the corresponding permutation of $T_2$ under the bijection $\phi$, i.e., $\pi'=\phi^{-1}(T_2)$.
	 According to Lemma~\ref{new-distribution}, the v-CIS and DRD length distributions of $\pi'$ are resp.~$\lambda$ and $\mu$.
	 The correspondence is apparently reversible, so the sets of (1) and (3) are of equal size.
	 Other pairs from the four sets can be done analogously, completing the proof.
\end{proof}

In the rest of the paper, we present some applications of Theorem~\ref{thm:equi}.
Recall that the Motzkin numbers $M_n$ can be defined by their generating function:
$$
\sum_{n \geq 0} M_n x^n=\frac {1-x-\sqrt{1-2x-3x^2}}{2x^2}.
$$
Elizalde and Mansour~\cite{sergi-toufik} showed that
the number of $132$-avoiding permutations in $\mathfrak{S}_n(132)$ that also avoid the consecutive pattern $123$ is counted by
the Motzkin number $M_n$.
Note that avoiding the consecutive pattern $123$ is tantamount to requiring all IR have length at most two.
As an immediate corollary of Theroem~\ref{thm:equi}, we present more sets of permutations which are counted by the Motzkin numbers.

\begin{corollary}[Motzkin family]\label{cor:motzkin}
	The following four sets of permutations are all counted by the Motzkin number $M_n$:
	\begin{itemize}
		
		\item[(1)] $\pi \in \mathfrak{S}_n(132)$ whose IR have length at most two;
		
		\item[(2)]
		$\pi \in \mathfrak{S}_n(132)$ whose LDE decreasing paths have length at most two;
		
		\item[(3)] $\pi \in \mathfrak{S}_n(132)$ whose v-CIS have length at most two;
		
		\item[(4)] $\pi \in \mathfrak{S}_n(132)$ whose DR have length at most two.
	\end{itemize}
\end{corollary}

We next present more enumerative results.
Some notation are needed to proceed.
An \emph{(unlabelled) set-alternating E-tree} is a plane tree where the even-level vertices carry (indistinguishable) labels from a set $E$
and the odd-level vertices carry (indistinguishable) labels from a set $O$.
\emph{O-trees} are defined similarly.

For $t \geq 0$, let $\kappa_t(n,m)$ denote the number of weak compositions of $n$ into $m\geq 0$ parts each of which
is no larger than $t$, i.e., $a_1+a_2+\cdots + a_m=n$ and $a_i$ is an integer satisfying $0\leq a_i \leq t$, where we
make the convention that $\kappa_t(0,0)=1$. Obviously, $\kappa_t(n,m)=0$ if $n<0$.

\begin{theorem}\label{thm:comp-gen}
	Suppose $p+q=n+1$. Then,
	the number of $\pi=\pi_1\pi_2 \cdots \pi_n \in \mathfrak{S}_n(132)$ with $p$ DR of which the one starting with $\pi_1$ has a length at most $h$ while each of the rest is of length no longer than $h+1$ and with
	$q$ v-CIS
	each of which is no longer than $l$ is given by
	\begin{align}\label{eq:comp-gen}
		\kappa_{l}(p-1,q)\kappa_{h}(q,p)-\Big\{\sum_{i=0}^{l-1}(i+1)\kappa_{l}(p-i-2,q-1)\Big \} \Big\{ \sum_{j=0}^{h-1}(j+1)\kappa_{h}(q-j-1,p-1)\Big \}.
	\end{align}
\end{theorem}
\proof Combining Theorem~\ref{thm:chen3} and Lemma~\ref{new-distribution}, the set of $132$-avoiding permutations under consideration
has the same size as the set of plane trees of $n$ edges where there are $p$ even-level vertices, each with outdegree at most $h$ and $q$ odd-level vertices, each with outdegree at most $l$.
 Let

\[\omega_{1}(t_{1}, t_{2})=\sum_{T\in\mathfrak{F}_{E}}t_{1}^{\# \text{vertices in $E$ in $T$}}t_{2}^{\# \text{vertices in $O$ in $T$}},\]
\[\omega_{2}(t_{1}, t_{2})=\sum_{T\in\mathfrak{F}_{O}}t_{1}^{\# \text{vertices in $E$ in $T$}}t_{2}^{\# \text{vertices in $O$ in $T$}},\]
where $\mathfrak{F}_{E}$ denotes the set of unlabelled set-alternating E-trees with every $E$-vertex having at most $h$ children, every $O$-vertex having at most $l$ children, while $\mathfrak{F}_{O}$ denotes the set of unlabelled set-alternating O-trees with every $E$-vertex having at most $h$ children, every $O$-vertex having at most $l$ children. Clearly, the number in question is $[t_{1}^{p} t_{2}^{q}]\omega_{1}$.

Note that the following relation is obvious
\[\omega_{1}=t_{1}\frac{1-\omega_{2}^{h+1}}{1-\omega_{2}},\quad \omega_{2}=t_{1}\frac{1-\omega_{1}^{l+1}}{1-\omega_{1}}.\]
Let
\[
g(x_{1}, x_{2})=x_{1},\quad f_{1}(x_{1}, x_{2})=\frac{1-x_{2}^{h+1}}{1-x_{2}},\quad f_{2}(x_{1}, x_{2})=\frac{1-x_{1}^{l+1}}{1-x_{1}}.
\]
Making use of the bivariate Lagrange inversion formula~\cite{BR,gessel}, we have

\begin{align*}
&[t_{1}^{p} t_{2}^{q}]\omega_{1}=[x_{1}^{p} x_{2}^{q}]g \cdot f_{1}^{p} \cdot f_{2}^{q}
\begin{bmatrix}
	{1-\frac{x_{1}}{f_{1}}\frac{\partial f_{1}}{\partial x_{1}}}&-\frac{x_{1}}{f_{2}}\frac{\partial f_{2}}{\partial x_{1}}\\
	-\frac{x_{2}}{f_{1}}\frac{\partial f_{1}}{\partial x_{2}}&{1-\frac{x_{2}}{f_{2}}\frac{\partial f_{2}}{\partial x_{2}}}
\end{bmatrix}\\
	=&[x_{1}^{p}x_{2}^{q}] \frac{x_1(1-x_{1}^{l+1})^{q}}{(1-x_{1})^{q}} \frac{(1-x_{2}^{h+1})^{p}}{(1-x_{2})^{p}}\\
	&\quad \times \left[1-\frac{x_{1}x_{2}}{(1-x_{1}^{l+1})(1-x_{2}^{h+1})}
	\Big(\frac{1-x_{1}^{l}}{1-x_{1}}-lx_{1}^{l}\Big)\Big(\frac{1-x_{2}^{h}}{1-x_{2}}-hx_{2}^{h}\Big)\right]\\
	=&\kappa_{l}(p-1,q)\kappa_{h}(q,p)-\sum_{i=0}^{l-1}(i+1)\kappa_{l}(p-i-2,q-1)\sum_{j=0}^{h-1}(j+1)\kappa_{h}(q-j-1,p-1),	
\end{align*}
where the last simplication follows from Lemma~\ref{lem:computation} in the Appendix.\qed

Let $h$ go to infinity, we obtain

\begin{corollary}
The number of $132$-avoiding permutations $\pi=  \mathfrak{S}_n(132)$  with $q$ IR each of which has a length
at most $l$ is given by
\begin{align}
\sum_{i = 0}^{l} \left\{ {n \choose q}- i {n \choose q-1}\right\} \kappa_{l}(n-q-i,q-1).
\end{align}
\end{corollary}
\proof When $h$ is large enough, we obviously have
$$
\kappa_{h}(n,m)={n+m-1 \choose m-1}.
$$
Then, eq.~\eqref{eq:comp-gen} reduces to
$$
{q+p-1 \choose p-1} \kappa_{l}(p-1,q) -\Big\{ \sum_{j = 0}^{q-1} (j+1) {q+p-3-j \choose p-2}\Big\} \Big\{ \sum_{i = 0}^{l-1} (i+1) \kappa_{l}(p-i-2,q-1)\Big\}.
$$
Next, it is not hard to show by generating functions that
$$
\sum_{k=r}^{n-s}{k \choose r}{n-k \choose s}={n+1 \choose r+s+1}.
$$
Then, we have
\begin{align*}
 &\sum_{j = 0}^{q-1} (j+1) {q+p-3-j \choose p-2} = \sum_{j = 0}^{q-1} {j+1 \choose 1}  {q+p-2-j-1 \choose p-2}
 =  {n \choose p}.
\end{align*}
Consequently, eq.~\eqref{eq:comp-gen} equals
\begin{align*}
& {q+p-1 \choose p-1} \kappa_{l}(p-1,q) - {n\choose p} \Big\{ \sum_{i = 0}^{l-1} (i+1) \kappa_{l}(p-i-2,q-1)\Big\}\\
=& \sum_{i = 0}^{l} \left\{ {n \choose p-1}- i {n \choose p}\right\} \kappa_{l}(p-1-i,q-1).
\end{align*}
Note that IR and v-CIS are equidistributed from Theorem~\ref{thm:equi}, then the proof follows.
\qed

The numbers obtained by setting $l=2$ provide a refinement of the Motzkin numbers.
To be specific, we have the following result.

\begin{corollary}
	The number of $132$-avoiding permutations in $\mathfrak{S}_n(132)$ that avoid the
	consecutive pattern $123$ and have $q$ descents is given by
\begin{align}
	\sum_{i = 0}^{2} \left\{ {n \choose q}- i {n \choose q-1}\right\} \kappa_{2}(n-q-i,q-1).
\end{align}
	Moreover, we have
	\begin{align}
		M_n=\sum_{q} \sum_{i = 0}^{2} \left\{ {n \choose q}- i {n \choose q-1}\right\} \kappa_{2}(n-q-i,q-1).
	\end{align}
\end{corollary}

Finally, we leave it to the interested reader to show $\kappa_{2}(n,m)$ can be explicitly computed as follows:
\begin{align}
	\kappa_{2}(n,m)=\sum_{i=0}^m {m \choose i}{m-i \choose 2m-2i-n}.
\end{align}
We remark that further enumerative results can be obtained and some of them may be found in the arXiv versions of the work.
IR and DR can be formulated into consecutive monotone patterns. Which patterns correspond to
v-CIS and LDE paths? IR and LDE (resp. DR and v-CIS) are somehow intertwined, and what their joint length distribution really reveal w.r.t.~the structure of $132$-avoiding permutations remains unclear. These problems are interesting and left for future
investigations.

\section*{Acknowledgments}
The authors thank an anonymous referee for comments.
 The work was partially supported by the Anhui Provincial Natural Science Foundation of China (No.~2208085MA02)
 and Overseas Returnee Support Project on Innovation and Entrepreneurship of Anhui Province (No.~11190-46252022001).

\appendix

\section{Computation in Theorem~\ref{thm:comp-gen}}

\begin{lemma}\label{lem:computation}
For $l>0$, $q> 0$ and $p$ being any integer, we have
	\begin{align*}
		[x_1^p ]x_{1} \cdot \frac{(1-x_{1}^{l+1})^{q}}{(1-x_{1})^{q}}\cdot \frac{x_{1}}{(1-x_{1}^{l+1})}
		\left(\frac{1-x_{1}^{l}}{1-x_{1}}-lx_{1}^{l}\right)=\sum_{i=0}^{l-1}(i+1)\kappa_{l}(p-i-2,q-1).
	\end{align*}
\end{lemma}
\proof
First, for $t\geq 0$ and $m\geq 0$, it is easy to see
\[\sum_{n\geq 0}\kappa_{t}(n,m)x^{n}=\sum_{n=-\infty}^{+\infty}\kappa_{t}(n,m)x^{n}=(1+x+\cdots+x^{t})^{m}=\left(\frac{1-x^{t+1}}{1-x}\right)^{m}.\]
Then, we have
\begin{align*}
    &[x_1^p ]x_{1} \cdot \frac{(1-x_{1}^{l+1})^{q}}{(1-x_{1})^{q}}\cdot \frac{x_{1}}{(1-x_{1}^{l+1})}
		\left(\frac{1-x_{1}^{l}}{1-x_{1}}-lx_{1}^{l}\right)\\
=	&[x_{1}^{p-2}]\frac{(1-x_{1}^{l+1})^{q}}{(1-x_{1})^{q}}\cdot \frac{1}{(1-x_{1}^{l+1})}
		\left(\frac{1-x_{1}^{l}}{1-x_{1}}-lx_{1}^{l}\right)\\
= &[x_{1}^{p-2}]\frac{(1-x_{1}^{l+1})^{q-1}}{(1-x_{1})^{q+1}}-[x_{1}^{p-l-2}]\frac{(1-x_{1}^{l+1})^{q-1}}{(1-x_{1})^{q+1}}-l[x_{1}^{p-l-2}]\frac{(1-x_{1}^{l+1})^{q-1}}{(1-x_{1})^{q}}\\
= &[x_{1}^{p-2}]\frac{(1-x_{1}^{l+1})^{q-1}}{(1-x_{1})^{q-1} (1-x_{1})^{2}}-[x_{1}^{p-l-2}]\frac{(1-x_{1}^{l+1})^{q-1}}{(1-x_{1})^{q-1} (1-x_{1})^{2}}-l[x_{1}^{p-l-2}]\frac{(1-x_{1}^{l+1})^{q-1}}{(1-x_{1})^{q} }\\
   = &\sum_{i=0}^{p-2} {i+1 \choose i}\kappa_{l}(p-i-2,q-1)-\sum_{i=0}^{p-l-2}{i+1 \choose i}\kappa_{l}(p-l-i-2,q-1)\\
      &\quad -l\sum_{i=0}^{p-l-2}\kappa_{l}(p-l-i-2,q-1)\\
   = &\sum_{i=0}^{p-2}(i+1)\kappa_{l}(p-i-2,q-1)-\sum_{i=0}^{p-l-2}(i+l+1)\kappa_{l}(p-l-i-2,q-1)\\
   = &\sum_{i=0}^{p-2}(i+1)\kappa_{l}(p-i-2,q-1)-\sum_{i=l}^{p-2}(i+1)\kappa_{l}(p-i-2,q-1)\\
    =&\sum_{i=0}^{l-1}(i+1)\kappa_{l}(p-i-2,q-1),
\end{align*}
and the proof follows.
\qed


\begin{thebibliography}{99}




	\bibitem{BR} E. A. Bender, L. B. Richmond, A multivariate Lagrange inversion formula for asymptotic calculations, Electron. J. Combin., 5(1) (1998), R33.
	
	\bibitem{callan22} D. Callan, A note on generalized Narayana numbers, arXiv:2205.08277, 2022.

	
	\bibitem{chen3} R. X. F. Chen, A new bijection between RNA secondary structures and plane trees and its consequences, Electron. J. Combin., 26(4) (2019), P4. 48.



\bibitem{bill} W. Y. C. Chen, A general bijective algorithm for trees, Proc. Natl. Acad. Sci. USA, 87 (1990), 9635--9639.


	\bibitem{CK} A. Claesson, S. Kitaev, Classification of bijections between $321$- and $132$-avoiding permutations,
	S\'{e}m. Lothar. Combin., 60 (2008), B60d.
	
	


    \bibitem{height72} N. G. de Bruijn, D. E. Knuth, S. O. Rice, The average height of planted plane trees, Graph Theory and Comput., (1972), 15--22.


    
    \bibitem{sergi-toufik} S. Elizalde, T. Mansour, Restricted Motzkin permutations, Motzkin paths, continued fractions, and Chebyshev polynomials,
    Discrete Math. 305 (2005), 170--189.


   \bibitem{eli-noy} S. Elizalde, M. Noy, Consecutive patterns in permutations, Adv. Appl. Math., 30 (2003), 110--125.

   \bibitem{height82} P. Flajolet, A. Odlyzko, The average height of binary trees and other simple trees, J. Comput. Syst. Sci., 25(2) (1982), 171--213.



	\bibitem{gessel} I. M. Gessel, A combinatorial proof of the multivariate Lagrange inversion formula, J. Combin. Theory Ser. A, 45 (1987), 178--195.
	

    \bibitem{jani-rieper} M. Jani, R. G. Rieper, Continued fractions and Catalan problems, Electron. J. Combin., 7 (2000), R45.






\bibitem{hemp83} R. Kemp, The average height of planted plane trees with M leaves, J. Combin. Theory Ser. B, 34(2)(1983), 191--208.




\bibitem{Knuth} D. Knuth, The art of computer programming, I: Fundamental algorithms, Addisom-Wesley Publishing Co., Reading Mass.-London-Don Mills. Ont., 1969.


    \bibitem{krattenthaler} C. Krattenthaler, Permutations with restricted patterns and Dyck paths, Adv. Appl. Math., 27 (2001), 510--530.

	

\bibitem{prodinger84} H. Prodinger, The average height of the second highest leaf of a planted
plane tree, European J. Combin., 5(4)(1984), 351--357.




	\bibitem{waterman} W. R. Schmitt, M. S. Waterman, Linear trees and RNA secondary structure, Discrete Appl. Math., 51(3) (1994), 317--323.


	\bibitem{waterman4} T. F. Smith, M. S. Waterman, RNA secondary structure, Math. Biol., 42 (1978), 31--49.
	


\bibitem{yzhuang} Y. Zhuang, Counting permutations by runs, J. Combin. Theory Ser. A, 142(2016), 147--176.

\end{thebibliography}
\end{document}